\documentclass[11pt]{article}
\usepackage{graphicx} % Required for inserting images
\usepackage[utf8]{inputenc}
\usepackage{titlesec}
\titleformat{\section}
{\normalfont\scshape}
{\thesection.}{0.5em}{\centering}
\titleformat{\subsection}
{\small\scshape}
{\thesubsection.}{0.5em}{\centering}
\titleformat{\subsubsection}
{\small\scshape}
{\thesubsubsection.}{0.5em}{\centering}
\usepackage{tikz-cd}
\usepackage{tikz}

\usepackage[mathcal]{euscript}

\usepackage[left=3.5cm,right=3.5cm,top=3.5cm,bottom=3.5cm]{geometry}
\usepackage{comment}
\usepackage{amsfonts}
\usepackage{indentfirst}
\renewcommand{\abstract}{\small{\section*{\abstractname}}}
\usepackage{hyperref}
\hypersetup{
    colorlinks,
    citecolor=blue,
    filecolor=blue,
    linkcolor=blue,
    urlcolor=blue
}
\usepackage{amsmath}
\usepackage{amssymb}
\usepackage{amsthm}
\usepackage{mathrsfs}
\newtheorem{theorem}{Theorem}
\newtheorem{mainthm}{Theorem}

\newtheorem{lemma}{Lemma}
\newtheorem{proposition}{Proposition}
\newtheorem{corollary}{Corollary}
\newtheorem{definition}{Definition}
\theoremstyle{remark}

\usepackage{stmaryrd}
\usepackage{quiver}
\usetikzlibrary{nfold}

\usepackage[
backend=biber,
style=numeric,
sorting=ynt
]{biblatex}
\title{\Large{\textbf{Complemented Copies of $c_{0}$ in Positive Tensor Products of Banach Lattices}}}
\author{\normalsize{\textsc{Vasily Melnikov}}}
\date{\normalsize{\textsc{November 2026}}}

\begin{document}

\maketitle
\begin{abstract}
     A result of Cembranos states that a non-trivial injective tensor product of a $C$-space contains a complemented copy of $c_{0}$, and in particular fails the Grothendieck property. We establish a positive analogue of the Cembranos theorem for tensor products of Banach lattices. If $E$ and $F$ are infinite dimensional Banach lattices, one of which contains a copy of $c_{0}$, then the Wittstock tensor product $E\widetilde{\otimes}_{\vert{\varepsilon}\vert}F$ contains a complemented copy of $c_{0}$. If $F$ is in addition reflexive, then $E\widetilde{\otimes}_{\vert{\varepsilon}\vert}F$ fails the positive Grothendieck property.
\end{abstract}
\section{Introduction}\label{sec:intro}
By a result of Cembranos \cite{cembranos}, an injective tensor product $C(K)\widetilde{\otimes}_{\varepsilon}E$ always contains a complemented copy of $c_{0}$, where $K$ is an infinite compact Hausdorff space, and $E$ is an infinite dimensional Banach space. In particular, $C(K)\widetilde{\otimes}_{\varepsilon}E$ is never a Grothendieck space.
\par
A number of results hint at a more general indecomposability phenomenon, wherein a Grothendieck space often cannot be written as a tensor product of infinite dimensional Banach spaces. We list some examples. In connection with the study of automorphisms of the Calkin algebra, a usually untractable problem (see \cite{weaver-phil-calkin,farah-calkin}), it was asked whether the Calkin algebra is a tensor product of $C^{\ast}$-algebras. Ghasemi \cite{ghasemi} resolved the question in the negative, a result which Kania \cite{kancstar} extended to any $C^{\ast}$-algebra which is a Grothendieck space. For general Banach spaces, it is known that the projective tensor product $E\widetilde{\otimes}_{\pi}F$ is Grothendieck only if $E$ or $F$ is reflexive (see \cite[Proposition 5.3.1]{grospace}), and conjectured that the injective tensor product $E\widetilde{\otimes}_{\varepsilon}F$ is not Grothendieck if it fails to be reflexive (see \cite[p. 1158]{czechgro}). Similar observations have been made for positive tensor products of Banach lattices. For instance, Bu, Craddock, and Ji \cite{ban-latt-gro-lphi,ban-latt-gro-lp} have shown that the Wittstock tensor product $E\widetilde{\otimes}_{\vert{\varepsilon}\vert}\ell_{\Phi}$ is never a Grothendieck space if $E$ is non-reflexive,\footnote{Specifically, they show that $E\widetilde{\otimes}_{\vert{\varepsilon}\vert}\ell_{\Phi}$ is a Grothendieck space iff $E$ is and every positive operator $\ell_{\Phi^{\ast}}\longrightarrow E^{\ast\ast}$ is compact. But if $E$ is a non-reflexive Banach lattice with the Grothendieck property, $E^{\ast\ast}$ contains a lattice copy of $\ell_{\infty}$ (see \cite[Proposition 3.2.7]{grospace}), so this latter condition can only be satisfied if $E$ is reflexive.} where $(\Phi,\Phi^{\ast})$ is a conjugate pair of $\Delta_{2}$-Young functions and $\ell_{\Phi}$ is the associated Orlicz sequence space. Similar questions have been addressed for the Fremlin tensor product, both of general Banach lattices (see \cite{kalton-regular-compact,reflex-regular}) and for one factor an Orlicz sequence space (see \cite{ban-latt-gro-lphi,ban-latt-gro-lp}).
\par
We formulate a general version of the indecomposability principle for tensor products of infinite dimensional Banach lattices $E$ and $F$. If $E$ contains $c_{0}$, then the Wittstock tensor product $E\widetilde{\otimes}_{\vert{\varepsilon}\vert}F$ contains a complemented copy of $c_{0}$ and thus is not a Grothendieck space.
\begin{mainthm}\label{thm:intro-main}
    Let $E$ and $F$ be infinite dimensional Banach lattices such that $E$ contains $c_{0}$. Then $E\widetilde{\otimes}_{\vert{\varepsilon}\vert}F$ contains a complemented copy of $c_{0}$ and in particular fails the Grothendieck property. If $F$ is in addition reflexive, then $E\widetilde{\otimes}_{\vert{\varepsilon}\vert}F$ fails the positive Grothendieck property.
\end{mainthm}
In fact, reflexivity in the last part of Theorem \ref{thm:intro-main} can be replaced with the weaker positive Josefson-Nissenzweig property introduced by Wójtowicz \cite{positive-jn}.
\par
An immediate consequence of Theorem \ref{thm:intro-main} conjoined with a result of Räbiger \cite{rabiger-gro-ban-latt} is the following.
\begin{mainthm}\label{thm:intro-corr}
    Let $E$ and $F$ be infinite dimensional Banach lattices with $E$ non-reflexive. Then $E\widetilde{\otimes}_{\vert{\varepsilon}\vert}F$ fails to be a Grothendieck space.
\end{mainthm}
Thus, for non-reflexive $E$ and infinite dimensional $F$, $E\widetilde{\otimes}_{\vert{\varepsilon}\vert}F$ fails the Grothendieck property—a vast generalization of earlier results for $F$ a reflexive Orlicz sequence space (see \cite{ban-latt-gro-lphi,ban-latt-gro-lp}).
\par
Our approach to these results is based on a new result on weakly unconditionally Cauchy series in the Wittstock tensor product.
\begin{mainthm}\label{thm:intro-wucs}
    Let $\sum_{n=1}^{\infty}e_{n}$ be a weakly unconditionally Cauchy series in $E$ with each $e_{n}$ positive, and let $(f_{n})_{n=1}^{\infty}$ be norm bounded. Then the series $\sum_{n=1}^{\infty}e_{n}\otimes f_{n}$ is weakly unconditionally Cauchy in $E\widetilde{\otimes}_{\vert{\varepsilon}\vert}F$.
\end{mainthm}
Theorem \ref{thm:intro-wucs} may be of independent interest; its conclusion is a property the Wittstock norm shares with the injective norm but which fails for the Fremlin and projective norms.
\par
The paper is structured as follows. Our main result, Theorem \ref{thm:intro-main}, is presented in \S\ref{sec:main} and proved in \S\ref{sec:proof-main}, while the instrumental Theorem \ref{thm:intro-wucs} is formulated and proved in \S\ref{sec:wuc-wittstock}.
\section{Notation and Preliminaries}\label{sec:notation-prelim}
For a Banach space $E$, we write $B_{E}$ for the unit ball, $S_{E}$ for the unit sphere, and $E^{\ast}$ for the dual space. If $F$ is another Banach space, the space of linear operators between $E$ and $F$ is denoted $\mathscr{L}(E,F)$; the operator norm of $T\in\mathscr{L}(E,F)$ is denoted $\Vert{T}\Vert$.
\par
A formal series $\sum_{n=1}^{\infty}x_{n}$ in $E$ is said to be weakly unconditionally Cauchy if
\begin{equation*}
    \sum_{n=1}^{\infty}\vert{\langle{x_{n},x^{\ast}}\rangle}\vert<\infty
\end{equation*}
for each $x^{\ast}\in E^{\ast}$. Equivalently, $\sum_{n=1}^{\infty}x_{n}$ is weakly unconditionally Cauchy if there exists a constant $C\geq0$ such that
\begin{equation*}
    \left\Vert{\sum_{n}\lambda_{n}x_{n}}\right\Vert_{E}\leq C\sup_{n}\vert{\lambda_{n}}\vert
\end{equation*}
for all finite sequences $(\lambda_{n})_{n}$ of real numbers (see \cite[Proposition 4.3.9]{megg}). Weak unconditional Cauchyness is weaker than unconditional convergence. For example, if $(x_{n})_{n=1}^{\infty}$ is a $c_{0}$ basis in $E$, $\sum_{n=1}^{\infty}x_{n}$ is weakly unconditionally Cauchy but not unconditionally convergent (see \cite[Theorem 4.3.10]{megg}).
\par
An operator $T\in\mathscr{L}(E,F)$ is said to be unconditionally converging if $T$ sends weakly unconditionally Cauchy series to unconditionally convergent series. A Banach space $E$ is said to have Pełczyński’s property (V) if every unconditionally converging operator from $E$ to any Banach space is weakly compact.
\par
If $E$ and $F$ are Banach lattices, and $T\in\mathscr{L}(E,F)$, $T$ is said to be positive if $T(E_{+})\subseteq F_{+}$; the set of positive operators is denoted $\mathscr{L}(E,F)_{+}$. The space of regular operators is denoted $\mathscr{L}_{r}(E,F)=\mathscr{L}(E,F)_{+} - \mathscr{L}(E,F)_{+}$. We define the regular norm on $\mathscr{L}_{r}(E,F)$ by
\begin{equation*}
    \Vert{T}\Vert_{r}=\inf\left\{\Vert{S}\Vert_{\mathrm{op}}:S\in\mathscr{L}(E,F)_{+},\vert{T(x)}\vert\leq S(\vert{x}\vert)\forall x\in E\right\},
\end{equation*}
for all $T\in\mathscr{L}_{r}(E,F)$.
\subsection{Tensor Products and Tensor Norms}
Given two normed spaces $E$ and $F$, the tensor product $E\otimes F$ can be equipped with a number of norms (see Ryan \cite{introtensor} for an overview of the theory). Given a norm $\alpha$ on $E\otimes F$, we denote by $E\otimes_{\alpha}F$ the normed space $(E\otimes F,\Vert{\cdot}\Vert_{\alpha})$, and by $E\widetilde{\otimes}_{\alpha}F$ the completion of $E\otimes_{\alpha} F$. We denote by $\varepsilon$ the injective norm, and by $\pi$ the projective norm.
\par
For Banach lattices $E$ and $F$, there are two natural norms on $E\otimes F$, the Fremlin and Wittstock norms, under which $E\widetilde{\otimes}_{\alpha}F$ is a Banach lattice. The Fremlin norm is denoted by $\vert{\pi}\vert$, and the Wittstock norm is denoted by $\vert{\varepsilon}\vert$. We refer the reader to Labuschagne \cite{order-reasonable} for their definitions and further details. The relation between the Wittstock and Fremlin norm on $E\otimes F$ is that $\vert{\varepsilon}\vert\leq\vert{\pi}\vert$.
\par
The Wittstock norm agrees with the regular norm. More precisely, the natural injection $E\otimes_{\vert{\varepsilon}\vert} F\longrightarrow\mathscr{L}_{r}(E^{\ast},F)$ is an isometric embedding.
\subsection{The Grothendieck Property}
Grothendieck \cite{groth} introduced the following property of Banach spaces.
\begin{definition}\label{def:gp}
    A Banach space $E$ has the Grothendieck property if every weak-star null sequence in $E^{\ast}$ is weakly null.
\end{definition}
Every reflexive space is trivially Grothendieck; the first non-reflexive example was given by Grothendieck \cite[Théorème 9]{groth}, who showed that $\ell_{\infty}$ has the Grothendieck property. Contrastingly, $c_{0}$ fails the Grothendieck property, as the sequence of coordinate functionals in $c^{\ast}_{0}$ is weak-star null but not weakly null. Because the Grothendieck property is inherited by complemented subspaces (see \cite[Proposition 3.1.4]{grospace}), it follows that no Grothendieck space contains a complemented copy of $c_{0}$. We refer the reader to González and Kania \cite{grospace} for further details on Grothendieck spaces.
\par
The following weakening of Definition \ref{def:gp} has been proposed for Banach lattices (see, e.g., \cite{posit-gro-property-subsets}).
\begin{definition}
    A Banach lattice $E$ has the positive Grothendieck property if every positive weak-star null sequence in $E^{\ast}$ is weakly null.
\end{definition}
For example, the space $c$ of convergent sequences has the positive Grothendieck property (see \cite[p. 6]{posit-gro-property-subsets}), but fails the Grothendieck property.
\subsection{The Positive Josefson-Nissenzweig Property}
The Josefson-Nissenzweig theorem states that in an infinite dimensional dual Banach space there is a weak-star null sequence in the unit sphere. In analogy with the Josefson-Nissenzweig theorem, the following property was introduced by Wójtowicz \cite{positive-jn}.
\begin{definition}
    A Banach lattice $E$ has the positive Josefson-Nissenzweig property, written $E\in\mathfrak{JN}^{+}$, if there exists a weak-star null $(x^{\ast}_{n})_{n=1}^{\infty}\subseteq S_{E^{\ast}}\cap E^{\ast}_{+}$.
\end{definition}
\begin{proposition}\label{prop:has-positive-jn}
    If $E$ can be mapped positively onto an infinite dimensional order continuous Banach lattice, then $E\in\mathfrak{JN}^{+}$. In particular, if $E$ is reflexive and infinite dimensional, then $E\in\mathfrak{JN}^{+}$.
\end{proposition}
\begin{proof}
    See \cite[Theorem 1]{positive-jn}.
\end{proof}
An easy way to generate Banach lattices in $\mathfrak{JN}^{+}$ is via direct sums: if $E$ is any Banach lattice, and $F\in\mathfrak{JN}^{+}$, then $E\oplus F\in\mathfrak{JN}^{+}$ (e.g., $\ell_{\infty}\oplus\ell_{2}$).
\section{Failure of the Grothendieck Property in Wittstock Tensor Products}\label{sec:main}
The main result of this article is the following.
\begin{theorem}\label{thm:main}
    Let $E$ and $F$ be infinite dimensional Banach lattices.
    \begin{enumerate}
        \item If $E$ contains $c_{0}$, then $E\widetilde{\otimes}_{\vert{\varepsilon}\vert}F$ contains a complemented copy of $c_{0}$.
        \item If $E$ contains $c_{0}$ and $F\in\mathfrak{JN}^{+}$, then $E\widetilde{\otimes}_{\vert{\varepsilon}\vert}F$ contains a complemented copy of $c_{0}$ and fails the positive Grothendieck property.
    \end{enumerate}
\end{theorem}
The proof of Theorem \ref{thm:main} is contained in \S\ref{sec:proof-main}. The proof relies on results from \S\ref{sec:wuc-wittstock}, some of which may be interesting in their own right.
\par
Both parts of Theorem \ref{thm:main} are sharp. If $E$ and $F$ are reflexive, it is possible for $E\widetilde{\otimes}_{\vert{\varepsilon}\vert}F$ to be reflexive (see Proposition \ref{prop:counter-main} below) and thus to contain no copies of $c_{0}$, let alone a complemented copy. As for the second part of the theorem, if $F\notin\mathfrak{JN}^{+}$, it is possible for $E\widetilde{\otimes}_{\vert{\varepsilon}\vert}F$ to satisfy the positive Grothendieck property, as the following example demonstrates.
\begin{proposition}\label{prop:c-spaces-pgp}
    Let $K$ and $L$ be infinite compact Hausdorff spaces. Then $C(K)\widetilde{\otimes}_{\vert{\varepsilon}\vert}C(L)$ contains a complemented copy of $c_{0}$, but satisfies the positive Grothendieck property.
\end{proposition}
Every infinite dimensional $C$-space contains a copy of $c_{0}$. Thus, for $E=C(K)$ and $F=C(L)$, every condition from the second part of Theorem \ref{thm:main} besides the positive Josefson-Nissenzweig property is met. Remark also that if $K=\beta\Gamma$, where $\beta\Gamma$ is the Stone-Čech compactification of a discrete space $\Gamma$, then $C(K)$ has the Grothendieck property.
\begin{proof}
    Because $C(K)\widetilde{\otimes}_{\vert{\varepsilon}\vert}C(L)\simeq C(K\times L)$ (see \cite[Corollary 3E]{fremlin-ban-lat}), $C(K)\widetilde{\otimes}_{\vert{\varepsilon}\vert}C(L)$ has a strong order unit and thus satisfies the positive Grothendieck property (see \cite[p. 6]{posit-gro-property-subsets}). On the other hand, $C(K\times L)$ contains a complemented copy of $c_{0}$ (see \cite{cembranos}).
\end{proof}
We now note some corollaries to Theorem \ref{thm:main}.
\begin{corollary}\label{corr:first}
    Let $E$ and $F$ be infinite dimensional Banach lattices such that $E$ contains $c_{0}$ and $F$ is reflexive. Then $E\widetilde{\otimes}_{\vert{\varepsilon}\vert}F$ fails the positive Grothendieck property, and contains a complemented copy of $c_{0}$.
\end{corollary}
\begin{proof}
    The result is a joint consequence of Proposition \ref{prop:has-positive-jn} and Theorem \ref{thm:main}.
\end{proof}
\begin{corollary}\label{corr:notgro}
    Let $E$ and $F$ be infinite dimensional Banach lattices with $E$ non-reflexive. Then $E\widetilde{\otimes}_{\vert{\varepsilon}\vert}F$ fails to be a Grothendieck space.
\end{corollary}
\begin{proof}
    Suppose for contradiction that $E\widetilde{\otimes}_{\vert{\varepsilon}\vert}F$ is Grothendieck but $E$ is non-reflexive. Being a complemented subspace of $E\widetilde{\otimes}_{\vert{\varepsilon}\vert}F$, $E$ is a Grothendieck space. According to \cite[Proposition 3.2.7]{grospace}, Banach lattices with the Grothendieck property have Pełczyński’s property (V). So $E$ is non-reflexive and has Pełczyński’s property (V), in which case $E$ contains $c_{0}$ (see \cite[Corollary III.3.3.C]{midealsalg}). Theorem \ref{thm:main} now implies that $E\widetilde{\otimes}_{\vert{\varepsilon}\vert}F$ contains a complemented copy of $c_{0}$, contradicting the Grothendieck property of $E\widetilde{\otimes}_{\vert{\varepsilon}\vert}F$.
\end{proof}
Corollary \ref{corr:notgro} cannot be improved to either remove the non-reflexivity condition or strengthen the conclusion to the existence of a complemented copy of $c_{0}$, as the following examples demonstrate.
\begin{proposition}\label{prop:counter-main}
    Fix $p,q\in(1,\infty)$. Then:
    \begin{enumerate}
        \item If $\frac{1}{p}+\frac{1}{q}>1$, the space $\ell_{p}\widetilde{\otimes}_{\vert{\varepsilon}\vert}\ell_{q}$ is reflexive, and in particular has the Grothendieck property.
        \item The space $\ell_{p}\widetilde{\otimes}_{\vert{\varepsilon}\vert}\ell_{1}$ satisfies the hypotheses of Corollary \ref{corr:notgro}, but does not contain any subspace isomorphic to $c_{0}$.
    \end{enumerate}
\end{proposition}
\begin{proof}
    The first point is a consequence of \cite[Example 4.2(2)]{type-tensor}. The second point is a consequence of \cite[Theorem 6.8(i)]{geom-bu}.
\end{proof}
\section{Weakly Unconditionally Cauchy Series in the Wittstock Norm}\label{sec:wuc-wittstock}
An important property of the injective tensor product $E\widetilde{\otimes}_{\varepsilon}F$ is that it preserves weakly unconditionally Cauchy series.
\begin{proposition}\label{prop:wuc-injective}
    Let $\sum_{n}e_{n}$ be a weakly unconditionally Cauchy series in $E$, and let $(f_{n})_{n=1}^{\infty}\subseteq B_{F}$. Then the series $\sum_{n=1}^{\infty}e_{n}\otimes f_{n}$ is weakly unconditionally Cauchy in $E\widetilde{\otimes}_{\varepsilon}F$.
\end{proposition}
\begin{proof}
Put
\begin{equation*}
    C=\sup_{N\in\mathbb{N}}\sup_{e^{\ast}\in B_{E^{\ast}}}\sum_{n=1}^{N}\vert{\langle{e_{n},e^{\ast}}\rangle\vert}<\infty,
\end{equation*}
which is finite by the closed graph theorem. For any finite sequence $(\lambda_{n})_{n}$ of real numbers, we have that
    \begin{equation*}
        \left\Vert{\sum_{n}\lambda_{n}e_{n}\otimes f_{n}}\right\Vert_{\varepsilon}=\sup_{e^{\ast}\in B_{E^{\ast}},f^{\ast}\in B_{F^{\ast}}}\left\vert\sum_{n}\lambda_{n}\langle{e_{n},e^{\ast}}\rangle\langle{f_{n},f^{\ast}}\rangle\right\vert
    \end{equation*}
    \begin{equation*}
        \leq\sup_{e^{\ast}\in B_{E^{\ast}},f^{\ast}\in B_{F^{\ast}}}\sum_{n}\left\vert{\lambda_{n}}{\langle{e_{n},e^{\ast}}\rangle}{\langle{f_{n},f^{\ast}}\rangle}\right\vert\leq C\sup_{n}\vert{\lambda_{n}}\vert.
    \end{equation*}
    Thus, $\sum_{n=1}^{\infty}e_{n}\otimes f_{n}$ is weakly unconditionally Cauchy in $E\widetilde{\otimes}_{\varepsilon}F$, as desired.
\end{proof}
The same is not true for the projective norm: if $(e_{n})_{n=1}^{\infty}$ are the coordinate vectors in $\mathbb{R}^{\mathbb{N}}$, then the tensor diagonal sequence $(e_{n}\otimes e_{n})_{n=1}^{\infty}$ is an $\ell_{1}$-basis in $c_{0}\widetilde{\otimes}_{\pi}\ell_{1}$ (see \cite[Proposition 5.2]{tensor-diag}). In view of \cite[Theorem 2B]{fremlin-ban-lat}, the same negative conclusion holds for the Fremlin tensor product $c_{0}\widetilde{\otimes}_{\vert{\pi}\vert}\ell_{1}$.
\par
We observe, under a positivity condition on the $e_{n}$'s, a similar result to Proposition \ref{prop:wuc-injective} for the Wittstock tensor norm.
\begin{theorem}\label{thm:wittstock-wucs}
    Let $\sum_{n}e_{n}$ be a weakly unconditionally Cauchy series in $E$ with each $e_{n}$ positive, and let $(f_{n})_{n=1}^{\infty}\subseteq B_{F}$. Then the series $\sum_{n=1}^{\infty}e_{n}\otimes f_{n}$ is weakly unconditionally Cauchy in $E\widetilde{\otimes}_{\vert{\varepsilon}\vert}F$.
\end{theorem}
\begin{proof}
    See \S\ref{subsec:proof-wucs}.
\end{proof}
The assumption in Theorem \ref{thm:wittstock-wucs} that the $e_{n}$'s are positive cannot be removed (see Proposition \ref{prop:sharp-wucs}).
\subsection{Lemmata}\label{subsec:lemmas-for-wucs}
We begin with the following observation.
\begin{lemma}\label{lem:isom}
    Let $j:E\otimes F\longrightarrow(E^{\ast}\widetilde{\otimes}_{\vert{\pi}\vert}F^{\ast})^{\ast}$ be the canonical map defined on simple tensors as
    \begin{equation*}
        j(x\otimes y)=\left(x^{\ast}\otimes y^{\ast}\longmapsto\langle{x,x^{\ast}}\rangle\langle{y,y^{\ast}}\rangle\right).
    \end{equation*}
    Then $j$ extends to an isometry $\iota:E\widetilde{\otimes}_{\vert{\varepsilon}\vert}F\longrightarrow(E^{\ast}\widetilde{\otimes}_{\vert{\pi}\vert}F^{\ast})^{\ast}$.
\end{lemma}
\begin{proof}
    Let $j^{\mathrm{bd}}:E^{\ast\ast}\otimes F^{\ast\ast}\longrightarrow(E^{\ast}\widetilde{\otimes}_{\vert{\pi}\vert}F^{\ast})^{\ast}$ be the canonical map defined on simple tensors as
    \begin{equation*}
        j^{\mathrm{bd}}(x^{\ast\ast}\otimes y^{\ast\ast})=\left(x^{\ast}\otimes y^{\ast}\longmapsto\langle{x^{\ast},x^{\ast\ast}}\rangle\langle{y^{\ast},y^{\ast\ast}}\rangle\right).
    \end{equation*}
    According to \cite[Theorem 2.2]{order-reasonable}, $j^{\mathrm{bd}}$ extends to an isometry $\iota^{\mathrm{bd}}:E^{\ast\ast}\widetilde{\otimes}_{\vert{\varepsilon}\vert}F^{\ast\ast}\longrightarrow(E^{\ast}\widetilde{\otimes}_{\vert{\pi}\vert}F^{\ast})^{\ast}$.
    \par
    Denote by $J_{E}$ and $J_{F}$ the bidual embeddings $E\subseteq E^{\ast\ast}$, $F\subseteq F^{\ast\ast}$. According to \cite[Theorem 5.1]{order-reasonable}, $J_{E}\otimes J_{F}:E\otimes F\longrightarrow E^{\ast\ast}\otimes F^{\ast\ast}$ extends to an isometry $\overline{J}:E\widetilde{\otimes}_{\vert{\varepsilon}\vert}F\longrightarrow E^{\ast\ast}\widetilde{\otimes}_{\vert{\varepsilon}\vert}F^{\ast\ast}$.
    \par
    Defining $\iota=\iota^{\mathrm{bd}}\circ\overline{J}$ yields the desired construction.
\end{proof}
An immediate corollary to Lemma \ref{lem:isom} is the following.
\begin{lemma}\label{lem:adjoint-surjective}
    For $\iota$ defined in Lemma \ref{lem:isom}, we have that the adjoint $\iota^{\ast}:(E^{\ast}\widetilde{\otimes}_{\vert{\pi}\vert}F^{\ast})^{\ast\ast}\longrightarrow(E\widetilde{\otimes}_{\vert{\varepsilon}\vert}F)^{\ast}$ of $\iota$ is surjective.
\end{lemma}
\begin{proof}
    The adjoint of an isometry is surjective (see \cite[Theorem 3.1.22]{megg}).
\end{proof}
Lemma \ref{lem:adjoint-surjective} can be used to prove results about linear functionals on $E\widetilde{\otimes}_{\vert{\varepsilon}\vert}F$. Every such linear functional can be realized as an element of $(E^{\ast}\widetilde{\otimes}_{\vert{\pi}\vert}F^{\ast})^{\ast\ast}$, in the sense that we have a canonical surjection $\iota^{\ast}:(E^{\ast}\widetilde{\otimes}_{\vert{\pi}\vert}F^{\ast})^{\ast\ast}\longrightarrow(E\widetilde{\otimes}_{\vert{\varepsilon}\vert}F)^{\ast}$. A property (P) about linear functionals on $E\widetilde{\otimes}_{\vert{\varepsilon}\vert}F$ can thus be proved via a ``lifting'' procedure: prove (P) for linear functionals realizable as elements of $E^{\ast}\widetilde{\otimes}_{\vert{\pi}\vert}F^{\ast}$, prove that (P) lifts from $E^{\ast}\widetilde{\otimes}_{\vert{\pi}\vert}F^{\ast}$ to the bidual $(E^{\ast}\widetilde{\otimes}_{\vert{\pi}\vert}F^{\ast})^{\ast\ast}$, and conclude from surjectivity of $\iota^{\ast}$ that (P) holds for all linear functionals on $E\widetilde{\otimes}_{\vert{\varepsilon}\vert}F$. Our proof of Theorem \ref{thm:wittstock-wucs} is based on this scheme; the first two steps are completed in Lemma \ref{lem:key-lemma}, while the last step is done in \S\ref{subsec:proof-wucs}. We resort to this strategy because the more straightforward approach of embedding $E^{\ast}\widetilde{\otimes}_{\vert{\pi}\vert}F^{\ast}$ in $(E\widetilde{\otimes}_{\vert{\varepsilon}\vert}F)^{\ast}$ fails if the bounded positive approximation property is not satisfied for both $E^{\ast}$ and $F^{\ast}$ (see \cite[Corollary 4.3]{blanco}).
\begin{lemma}\label{lem:key-lemma}
Denote by $J:E^{\ast}\widetilde{\otimes}_{\vert{\pi}\vert}F^{\ast}\longrightarrow(E^{\ast}\widetilde{\otimes}_{\vert{\pi}\vert}F^{\ast})^{\ast\ast}$ the bidual embedding.
\begin{enumerate}
    \item The composition $\iota^{\ast}\circ J$ is a positive linear map, and for each simple tensor $x^{\ast}\otimes y^{\ast}\in E^{\ast}\otimes F^{\ast}$,
    \begin{equation*}
        \iota^{\ast}(J(x^{\ast}\otimes y^{\ast}))=(x\otimes y\longmapsto\langle{x,x^{\ast}}\rangle\langle{y,y^{\ast}}\rangle).
    \end{equation*}
    \item There exists an absolute constant $C\geq0$ such that, for
    \begin{equation*}
        S^{N}=\left(u\longmapsto(\langle{e_{n}\otimes f_{n},\iota^{\ast}(u)}\rangle)_{n=1}^{N}\right),
    \end{equation*}
    we have that $\Vert{S^{N}:(E^{\ast}\widetilde{\otimes}_{\vert{\pi}\vert}F^{\ast})^{\ast\ast}\to\ell^{N}_{1}\Vert}\leq C$ for each $N\in\mathbb{N}$. Consequently, the assignment
\begin{equation*}
    S=\left(u\longmapsto(\langle{e_{n}\otimes f_{n},\iota^{\ast}(u)}\rangle)_{n=1}^{\infty}\right)
\end{equation*}
defines a bounded linear map $S:(E^{\ast}\widetilde{\otimes}_{\vert{\pi}\vert}F^{\ast})^{\ast\ast}\longrightarrow\ell_{1}$.
    \end{enumerate}
\end{lemma}
\begin{proof}
    For the first claim, note that
    \begin{equation*}
        \langle{x\otimes y,\iota^{\ast}(J(x^{\ast}\otimes y^{\ast}))}\rangle=\langle{\iota(x\otimes y),J(x^{\ast}\otimes y^{\ast})}\rangle=\langle{x^{\ast}\otimes y^{\ast},\iota(x\otimes y)}\rangle=\langle{x,x^{\ast}}\rangle\langle{y,y^{\ast}}\rangle,
    \end{equation*}
    as desired.
    \par
    We now move to the second claim. Put
    \begin{equation*}
        C=\sup_{N\in\mathbb{N}}\sup_{x^{\ast}\in B_{E^{\ast}}}{\sum_{n=1}^{N}}\vert{\langle{e_{n},x^{\ast}}\rangle}\vert<\infty,
    \end{equation*}
    which is finite by the closed graph theorem. Denoting by $R^{N}:(E\widetilde{\otimes}_{\vert{\varepsilon}\vert}F)^{\ast}\longrightarrow\ell^{N}_{1}$ the mapping
    \begin{equation*}
        R^{N}=\left(v\longmapsto(\langle{e_{n}\otimes f_{n},v}\rangle)_{n=1}^{N}\right),
    \end{equation*}
    it is clear that $R^{N}$ is weak-star continuous. Being an adjoint, $\iota^{\ast}$ is weak-star to weak-star continuous (see \cite[Theorem 3.1.11]{megg}). Thus, because $S^{N}=R^{N}\circ\iota^{\ast}$ is a composition of weak-star continuous operators, $S^{N}$ is weak-star continuous. Hence $S^{N}=\left(S^{N}\circ J\right)^{\ast\ast}$, and the estimate $\Vert{S^{N}:(E^{\ast}\widetilde{\otimes}_{\vert{\pi}\vert}F^{\ast})^{\ast\ast}\to\ell^{N}_{1}\Vert}\leq C$ is equivalent to $\Vert{S^{N}\circ J:E^{\ast}\widetilde{\otimes}_{\vert{\pi}\vert}F^{\ast}\to\ell^{N}_{1}\Vert}\leq C$. It thus suffices to show that
    \begin{equation}\label{eq:aprioriweaker}
    \forall N\in\mathbb{N}:\sum_{n=1}^{N}\vert{\langle{e_{n}\otimes f_{n},\iota^{\ast}(J(u))}\rangle}\vert\leq C\left(\Vert{u}\Vert_{\vert{\pi}\vert}+1\right)
    \end{equation}
    for all $u\in E^{\ast}\widetilde{\otimes}_{\vert{\pi}\vert}F^{\ast}$.
    \par
Applying \cite[Proposition 6.1]{schauder-fremlin}, there exists $(x^{\ast}_{n})_{n=1}^{\infty}\subseteq E^{\ast}_{+}\cap S_{E^{\ast}}$, $(y^{\ast}_{n})_{n=1}^{\infty}\subseteq F^{\ast}_{+}\cap S_{F^{\ast}}$, and $\lambda\in \ell_{1}\cap \mathbb{R}^{\mathbb{N}}_{+}$ such that
\begin{equation*}
    \Vert{\lambda}\Vert_{\ell_{1}}=\sum_{n=1}^{\infty}\lambda_{n}=\sum_{n=1}^{\infty}\lambda_{n}\Vert{x^{\ast}_{n}}\Vert_{E^{\ast}}\Vert{y^{\ast}_{n}}\Vert_{F^{\ast}}\leq\Vert{u}\Vert_{\vert{\pi}\vert}+1
\end{equation*}
and
\begin{equation*}
    \vert{u\vert}\leq v,
\end{equation*}
where $v=\sum_{n=1}^{\infty}\lambda_{n}x^{\ast}_{n}\otimes y^{\ast}_{n}$. The first claim implies $\iota^{\ast}(J(v))$ acts as
\begin{equation*}
    \langle{x\otimes y,\iota^{\ast}(J(v))}\rangle=\sum_{i=1}^{\infty}\lambda_{i}\langle{x,x^{\ast}_{i}}\rangle\langle{y,y^{\ast}_{i}}\rangle
\end{equation*}
on simple tensors $x\otimes y\in E\otimes F$. In particular, for each $N\in\mathbb{N}$,
\begin{equation*}
    \sum_{n=1}^{N}\vert{\langle{e_{n}\otimes f_{n},\iota^{\ast}(J(u))}\rangle}\vert\leq\sum_{n=1}^{N}{\langle{\vert{e_{n}\otimes f_{n}\vert},\iota^{\ast}(J(v))}\rangle}=\sum_{n=1}^{N}{\langle{e_{n}\otimes\vert{f_{n}\vert},\iota^{\ast}(J(v))}\rangle}
\end{equation*}
\begin{equation*}
    =\sum_{n=1}^{N}\sum_{i=1}^{\infty}\lambda_{i}\langle{e_{n},x^{\ast}_{i}}\rangle\langle{\vert{f_{n}\vert},y^{\ast}_{i}}\rangle\leq\sum_{n=1}^{N}\sum_{i=1}^{\infty}\lambda_{i}\langle{e_{n},x^{\ast}_{i}}\rangle
\end{equation*}
\begin{equation}\label{eq:last-pi1}
    =\sum_{n=1}^{N}\Vert{U(e_{n})}\Vert_{\ell_{1}}\leq\pi_{1}(U)C
\end{equation}
where $U:E\longrightarrow\ell_{1}$ is defined by $U(x)=(\lambda_{n}\langle{x,x^{\ast}_{n}}\rangle)_{n=1}^{\infty}$ for each $x\in E$, $\pi_{1}$ denotes the $1$-summing norm (allowed to take the value $\infty$), and we use \cite[p. 245]{order-reasonable} together with positivity of the $e_{n}$'s to deduce that $\vert{e_{n}\otimes f_{n}}\vert=e_{n}\otimes \vert{f_{n}}\vert$ for each $n$. But obviously $\pi_{1}(U)\leq\Vert{\lambda}\Vert_{\ell_{1}}\leq\Vert{u}\Vert_{\vert{\pi}\vert}+1<\infty$, and therefore (\ref{eq:last-pi1}) yields
\begin{equation*}
    \sum_{n=1}^{N}\vert{\langle{e_{n}\otimes f_{n},\iota^{\ast}(J(u))}\rangle}\vert\leq C\left(\Vert{u}\Vert_{\vert{\pi}\vert}+1\right)
\end{equation*}
for all $N\in\mathbb{N}$ and $u\in E^{\ast}\widetilde{\otimes}_{\vert{\pi}\vert}F^{\ast}$, as desired.
\end{proof}
\subsection{The Proof of Theorem \ref{thm:wittstock-wucs}}\label{subsec:proof-wucs}
\begin{proof}[Proof of Theorem \ref{thm:wittstock-wucs}]
    Define a linear map $R:(E\widetilde{\otimes}_{\vert{\varepsilon}\vert}F)^{\ast}\longrightarrow\ell_{\infty}$ by
    \begin{equation*}
        R(u)=(\langle{e_{n}\otimes f_{n},u}\rangle)_{n=1}^{\infty}.
    \end{equation*}
    The series $\sum_{n=1}^{\infty}e_{n}\otimes f_{n}$ is weakly unconditionally Cauchy iff, viewing $\ell_{1}$ as an algebraic subspace of $\ell_{\infty}$, we have that $\mathrm{ran}(R)\subseteq\ell_{1}$.
    \par
    Lemma \ref{lem:key-lemma} implies that the diagram
    \[\begin{tikzcd}
	{(E\widetilde{\otimes}_{\vert{\varepsilon}\vert}F)^{\ast}} && {(E^{\ast}\widetilde{\otimes}_{\vert{\pi}\vert}F^{\ast})^{\ast\ast}} \\
	\\
	{\ell_{\infty}} && {\ell_{1}}
	\arrow["R", from=1-1, to=3-1]
	\arrow["{\iota^{\ast}}", from=1-3, to=1-1]
	\arrow["{S}"', from=1-3, to=3-3]
	\arrow["\supseteq"', from=3-3, to=3-1]
\end{tikzcd}\]
    commutes. Since $\iota^{\ast}$ is surjective (see Lemma \ref{lem:adjoint-surjective}), this implies that $\mathrm{ran}(R)\subseteq\ell_{1}$, as desired.
\end{proof}
\subsection{A Counterexample}
Theorem \ref{thm:wittstock-wucs} assumes the $e_{n}$'s are positive. This assumption is invoked only once (at the beginning of equation (\ref{eq:last-pi1})), but as we shall now see, it cannot be removed in the context of Theorem \ref{thm:wittstock-wucs}. Let $(\Omega,\mathscr{F},\mathbb{P})$ be a non-atomic probability space, and denote by $L^{p}=L^{p}(\Omega,\mathscr{F},\mathbb{P})$ the associated space of $p$-integrable random variables (here $1\leq p<\infty$).
\begin{proposition}\label{prop:sharp-wucs}
    Let $p\in[1,\infty)$ and $q\in[1,2)$. There exists a sequence $(e_{n})_{n=1}^{\infty}\subseteq L^{p}$ with $\sum_{n}e_{n}$ unconditionally convergent and $(f_{n})_{n=1}^{\infty}\subseteq B_{\ell_{q}}$ such that the following holds. The series $\sum_{n}e_{n}\otimes f_{n}$ is not weakly unconditionally Cauchy in $L^{p}\widetilde{\otimes}_{\vert{\varepsilon}\vert}\ell_{q}$.
\end{proposition}
A key aspect of our construction is the use of Rademacher random variables. A random variable $\varepsilon$ on $(\Omega,\mathscr{F},\mathbb{P})$ has the Rademacher distribution if
\begin{equation*}
    \mathbb{P}(\{\varepsilon=1\})=\mathbb{P}(\{\varepsilon=-1\})=\frac{1}{2}.
\end{equation*}
As a consequence of non-atomicity, there exists an i.i.d. sequence $(\varepsilon_{n})_{n=1}^{\infty}$ of Rademacher random variables on $(\Omega,\mathscr{F},\mathbb{P})$. In what follows, we fix such a sequence $(\varepsilon_{n})_{n=1}^{\infty}$.
\begin{proof}
    Take any positive sequence $(a_{n})_{n=1}^{\infty}\in\ell_{2}$ which is not contained in $\ell_{q}$. Defining $e_{n}=a_{n}\varepsilon_{n}$ for each $n\in\mathbb{N}$, the Khintchine inequality implies the associated series $\sum_{n=1}^{\infty}e_{n}$ is unconditionally convergent in $L^{p}$.
    \par
    Let $(f_{n})_{n=1}^{\infty}\subseteq B_{\ell_{q}}$ be the sequence of coordinate vectors. According to \cite[Theorem 5.2]{geom-bu}, $L^{p}\widetilde{\otimes}_{\vert{\varepsilon}\vert}\ell_{q}$ is a sublattice of $(L^{p})^{\mathbb{N}}$, so that the lattice operations in the former may be taken coordinate-wise. In particular,
    \begin{equation*}
        \left\vert{\sum_{n=1}^{N}e_{n}\otimes f_{n}}\right\vert=\sum_{n=1}^{N}\vert{e_{n}}\vert\otimes f_{n}=\sum_{n=1}^{N}\left(a_{n}\mathbf{1}_{\Omega}\right)\otimes f_{n}=\mathbf{1}_{\Omega}\otimes\left(\sum_{n=1}^{N}a_{n}f_{n}\right)
    \end{equation*}
    in $L^{p}\widetilde{\otimes}_{\vert{\varepsilon}\vert}\ell_{q}$. Thus, since the Wittstock norm is a reasonable cross norm and a lattice norm,
    \begin{equation*}
        \left\Vert{\sum_{n=1}^{N}e_{n}\otimes f_{n}}\right\Vert_{\vert{\varepsilon}\vert}=\Vert{\mathbf{1}_{\Omega}}\Vert_{L^{p}}\left\Vert{\sum_{n=1}^{N}a_{n}f_{n}}\right\Vert_{\ell_{q}}=\left(\sum_{n=1}^{N}\vert{a_{n}\vert}^{q}\right)^{1/q}\to\infty
    \end{equation*}
    as $N\to\infty$, because $(a_{n})_{n=1}^{\infty}\notin\ell_{q}$ by stipulation. In particular, the Banach-Steinhaus theorem shows that $\sum_{n=1}^{\infty}e_{n}\otimes f_{n}$ cannot be weakly unconditionally Cauchy in $L^{p}\widetilde{\otimes}_{\vert{\varepsilon}\vert}\ell_{q}$, as desired.
\end{proof}
In Proposition \ref{prop:sharp-wucs}, it is important that $q<2$. Indeed, suppose $E$ is a $\mathcal{L}_{1}$-space and $F$ a Hilbert space (corresponding to the choice $p=1$ and $q=2$), in which case every linear operator from $E$ to $F^{\ast}$ is $1$-summing (see \cite[Corollary 6.29]{introtensor}). It follows that if the series associated to $(e_{n})_{n=1}^{\infty}\subseteq E$ is weakly unconditionally Cauchy, and $(f_{n})_{n=1}^{\infty}\subseteq F$ is bounded, then $\sum_{n}e_{n}\otimes f_{n}$ is weakly unconditionally Cauchy in $E\widetilde{\otimes}_{\pi}F$ and hence also in every other reasonable crossnorm.
\section{Proof of Theorem \ref{thm:main}}\label{sec:proof-main}
\begin{proof}[Proof of Theorem \ref{thm:main}]
    We begin with the first part of Theorem \ref{thm:main}. Because $E$ contains a $c_{0}$ isomorph, $E$ contains $c_{0}$ as a closed vector sublattice $E_{0}\subseteq E$ (see Theorem 2.5.6 and Theorem 2.4.12 of \cite{meyer-nie}). Let $(e_{n})_{n=1}^{\infty}\subseteq E_{+}\cap E_{0}$ be the corresponding $c_{0}$ basis. In view of Lotz's theorem \cite{lotz}, we may find positive extensions $(e^{\ast}_{n})_{n=1}^{\infty}\subseteq E^{\ast}_{+}$ of the coordinate functionals of $(e_{n})_{n=1}^{\infty}$, with $(e^{\ast}_{n})_{n=1}^{\infty}$ an $M$-bounded sequence for some $M>0$.
    \par
    By the Josefson-Nissenzweig theorem, there exists a weak-star null $(f^{\ast}_{n})_{n=1}^{\infty}\subseteq S_{F^{\ast}}$. Using \cite[Remark III.1]{rotsenthalwill}, we may find a sequence $(f_{n})_{n=1}^{\infty}\subseteq 2B_{F}$ satisfying the biorthogonality condition $\langle{f_{n},f^{\ast}_{m}}\rangle=\delta_{n,m}$.
    \par
    Define an operator $T:E\widetilde{\otimes}_{\vert{\varepsilon}\vert}F\longrightarrow\ell_{\infty}$ by
    \begin{equation*}
        T(u)=\left(\langle{u,e^{\ast}_{n}\otimes f^{\ast}_{n}}\rangle\right)_{n=1}^{\infty}\in\ell_{\infty}
    \end{equation*}
    for all $u\in E\widetilde{\otimes}_{\vert{\varepsilon}\vert}F$. For each $x\in E$ and $y\in F$,
    \begin{equation*}
        0\leq\limsup_{n\to\infty}\vert{T(x\otimes y)}_{n}\vert\leq\limsup_{n\to\infty}\vert{\langle{x,e^{\ast}_{n}}\rangle}{\langle{y,f^{\ast}_{n}}\rangle}\vert\leq M\Vert{x}\Vert_{E}\limsup_{n\to\infty}\vert{{\langle{y,f^{\ast}_{n}}\rangle}}\vert=0,
    \end{equation*}
    and $T$ must therefore take values in $c_{0}\subseteq \ell_{\infty}$.
    \par
    In view of Theorem \ref{thm:wittstock-wucs}, $\sum_{n}e_{n}\otimes f_{n}$ is weakly unconditionally Cauchy in $E\widetilde{\otimes}_{\vert{\varepsilon}\vert}F$. Biorthogonality implies $(T(e_{n}\otimes f_{n}))_{n=1}^{\infty}$ is the canonical $c_{0}$-basis, and therefore $T$ cannot be unconditionally converging. In view of \cite[Lemma 2]{polynom-uc}, $E\widetilde{\otimes}_{\vert{\varepsilon}\vert}F$ thus contains a complemented copy of $c_{0}$.
    \par
    As for the second part of Theorem \ref{thm:main}, remark that if $F\in\mathfrak{JN}^{+}$, then we can choose for the $f^{\ast}_{n}$'s above to be positive. Thus, because the $e^{\ast}_{n}$'s and the $f^{\ast}_{n}$'s are contained in $E^{\ast}_{+}$ and $F^{\ast}_{+}$ respectively, $T$ is positive. If $E\widetilde{\otimes}_{\vert{\varepsilon}\vert}F$ had the positive Grothendieck property, every positive operator $E\widetilde{\otimes}_{\vert{\varepsilon}\vert}F\longrightarrow c_{0}$ would be weakly compact (see \cite[p. 6]{posit-gro-property-subsets}), so $T$ would be weakly compact. The Orlicz-Pettis theorem now implies $T$ is unconditionally converging, but this contradicts the fact that $T$ is not unconditionally converging.
\end{proof}
\printbibliography

@article{cembranos,
    author = {Pilar Cembranos},
    title = {{${C(K,E)}$ Contains a Complemented Copy of $c_{0}$}},
    journal = {Proceedings of the American Mathematical Society},
    year = {1984},
    volume = {91},
    pages = {556-558}
}

@book{introtensor,
    author = {Raymond Ryan},
    title = {{Introduction to Tensor Products of {Banach} Spaces}},
    publisher = {Springer},
    year = {2002}
}

@article{kancstar,
    author = {Tomasz Kania},
    title = {{On ${C^{\ast}}$-Algebras Which Cannot be Decomposed Into Tensor Products With Both Factors Infinite-Dimensional}},
    journal = {The Quarterly Journal of Mathematics},
    year = {2015},
    volume = {66},
    pages = {1063-1068}
}

@article{grospace,
    author = {Manuel González and Tomasz Kania},
    title = {{Grothendieck Spaces: The Landscape and Perspectives}},
    journal = {Japanese Journal of Mathematics},
    year = {2021},
    volume = {16},
    pages = {247-313}
}

@article{czechgro,
    author = {Donghai Ji and Xiaoping Xue and Qingying Bu},
    title = {{The {Grothendieck} Property for Injective Tensor Products of {Banach} Spaces}},
    journal = {Czechoslovak Mathematical Journal},
    year = {2010},
    volume = {60},
    pages = {1153-1159}
}

@article{rotsenthalwill,
    author = {William Johnson and Haskell Rosenthal},
    title = {{On $w^{\ast}$-Basic Sequences and Their Applications to the Study of {Banach} Spaces}},
    journal = {Studia Mathematica},
    year = {1972},
    volume = {43},
    pages = {77-92}
}

@book{rabiger-gro-ban-latt,
    author = {Frank Räbiger},
    title = {{Beiträge Zur Strukturtheorie der Grothendieck-Räume}},
    publisher = {Springer},
    year = {1985}
}

@article{farah-calkin,
    author = {Ilijas Farah},
    title = {{All Automorphisms of the Calkin Algebra Are Inner}},
    journal = {Annals of Mathematics},
    year = {2011},
    volume = {173},
    pages = {619-661}
}

@article{weaver-phil-calkin,
    author = {Christopher Phillips and Nik Weaver},
    title = {{The Calkin Algebra Has Outer Automorphisms}},
    journal = {Duke Mathematical Journal},
    year = {2007},
    volume = {139},
    pages = {185-202}
}

@article{ghasemi,
    author = {Saeed Ghasemi},
    title = {{$SAW^{\ast}$-Algebras Are Essentially Non-Factorizable}},
    journal = {Glasgow Mathematical Journal},
    year = {2015},
    volume = {57},
    pages = {1-5}
}

@article{ban-latt-gro-lp,
    author = {Qingying Bu and Michelle Craddock and Donghai Ji},
    title = {{Reflexivity and the Grothendieck Property for Positive Tensor Products of Banach Lattices-I}},
    journal = {Positivity},
    year = {2010},
    volume = {14},
    pages = {59-68}
}

@article{ban-latt-gro-lphi,
    author = {Qingying Bu and Michelle Craddock and Donghai Ji},
    title = {{Reflexivity and the Grothendieck Property for Positive Tensor Products of Banach Lattices-II}},
    journal = {Quaestiones Mathematicae},
    year = {2009},
    volume = {32},
    pages = {339-350}
}

@article{order-reasonable,
    author = {Coenraad C. A. Labuschagne},
    title = {{Riesz Reasonable Cross Norms on Tensor Products of Banach Lattices}},
    journal = {Quaestiones Mathematicae},
    year = {2004},
    volume = {27},
    pages = {243-266}
}

@article{kalton-regular-compact,
    author = {Qingying Bu},
    title = {{On Kalton's Theorem for Regular Compact Operators and Grothendieck Property for Positive Projective Tensor Products}},
    journal = {Proceedings of the American Mathematical Society},
    year = {2020},
    pages = {2459-2467},
    volume = {148}
}

@article{reflex-regular,
    author = {Yongjin Li and Qingying Bu},
    title = {{Reflexivity for Spaces of Regular Operators on Banach Lattices}},
    journal = {Proceedings of the American Mathematical Society},
    year = {2022},
    volume = {150},
    pages = {4811-4818}
}

@article{geom-bu,
    author = {Qingying Bu and Ngai-Ching Wong},
    title = {{Some Geometric Properties Inherited by the Positive Tensor Products of Atomic Banach Lattices}},
    journal = {Indagationes Mathematicae},
    year = {2012},
    volume = {23},
    pages = {199-213}
}

@article{type-tensor,
    author = {Qingying Bu and Patrick N. Dowling},
    title = {{On Type and Cotype of Tensor Products of Banach Spaces}},
    journal = {Journal of Mathematical Analysis and Applications},
    year = {2026},
    volume = {553}
}

@article{fremlin-ban-lat,
    author = {David Fremlin},
    title = {{Tensor Products of Banach Lattices}},
    journal = {Mathematische Annalen},
    year = {1974},
    pages = {87-106},
    volume = {211}
}

@article{posit-gro-property-subsets,
    author = {Pablo Galindo and Vinícius C.C. Miranda},
    title = {{Grothendieck-Type Subsets of Banach Lattices}},
    journal = {Journal of Mathematical Analysis and Applications},
    year = {2022},
    volume = {506}
}

@article{blanco,
    author = {Ariel Blanco},
    title = {{On the Positive Approximation Property}},
    journal = {Positivity},
    year = {2016},
    pages = {719-742},
    volume = {20}
}

@article{schauder-fremlin,
    author = {Qingying Bu and Gerard Buskes},
    title = {{Schauder Decompositions and the Fremlin Projective Tensor Product of Banach Lattices}},
    journal = {Journal of Mathematical Analysis and Applications},
    year = {2009},
    pages = {335-351},
    volume = {355}
}

@book{meyer-nie,
    author = {Peter Meyer-Nieberg},
    title = {{Banach Lattices}},
    publisher = {Springer},
    year = {1991}
}

@article{lotz,
    author = {Heinrich Lotz},
    title = {{Extensions and Liftings of Positive Linear Mappings on Banach Lattices}},
    journal = {Transactions of the American Mathematical Society},
    year = {1975},
    pages = {85-100},
    volume = {211}
}

@article{positive-jn,
    author = {Marek Wójtowicz},
    title = {{A Banach-Lattice Version of the Josefson-Nissenzweig Theorem}},
    journal = {Indagationes Mathematicae},
    year = {2007},
    pages = {479-484},
    volume = {18}
}

@article{groth,
    author = {Alexander Grothendieck},
    title = {{Sur les Applications Linéaires Faiblement Compactes d’Espaces du type $C(K)$}},
    journal = {Canadian Journal of Mathematics},
    year = {1953},
    volume = {5},
    pages = {129-173}
}

@article{tensor-diag,
    author = {James Holub},
    title = {{Tensor Product Bases and Tensor Diagonals}},
    journal = {Transactions of the American Mathematical Society},
    year = {1970},
    volume = {151},
    pages = {563-579}
}

@article{polynom-uc,
    author = {Manuel González and Joaquín M. Gutiérrez},
    title = {{When Every Polynomial is Unconditionally converging}},
    journal = {Archiv der Mathematik},
    year = {1994},
    volume = {63},
    pages = {145-151}
}

@book{megg,
    author = {Robert E. Megginson},
    title = {An Introduction to Banach Space Theory},
    publisher = {Springer},
    year = {1998}
}

@book{midealsalg,
    author = {Peter Harmand and Dirk Werner and Wend Werner},
    title = {M-Ideals in Banach Spaces and Banach Algebras},
    publisher = {Springer},
    year = {1993}
}
\end{document}